\documentclass[11pt]{article}
 \usepackage{amssymb}
 \usepackage{amsmath,amsfonts,amscd,amsthm}

 \usepackage[left=1in,right=0.5in,top=1in,bottom=1in]{geometry}
 \usepackage{blindtext}

\newtheorem{theorem}{Theorem}
\theoremstyle{plain}

\newtheorem{corollary}{Corollary}

\newtheorem{remark}{Remark}

\numberwithin{equation}{section}
 
\begin{document}
 \begin{center}
 \textbf{\large Certain composition formulae for the fractional integral operators}\\[4mm]
 { Praveen Agarwal$^{1, \ast}$ and Priyanka Harjule$^{2}$ }\\[2mm]

 $^{1}$Department of Mathematics,
Anand International College of Engineering,\\

Jaipur 303012, Rajasthan, India\\

\textbf{E-Mail: praveen.agarwal@anandice.ac.in\;\; goyal.praveen2011@gmail.com}\\[1mm]

$^{2}$ Department of Mathematics, Malaviya National Institute of Technology,\\[0pt]
 Jaipur 302017, Rajasthan, India \\[1mm]

 \textbf{E-Mail: priyankaharjule5@gmail.com}\\[6mm]

 \textbf{Abstract}
 \begin{quotation}
 In this paper we establish some (presumably new) interesting expressions for the composition of some well known fractional integral operators $ I^{\mu}_{a+}, D^{\mu}_{a+}  $,$ I^{\gamma , \mu}_{a+}$ and also derive an integral operator $\mathcal{H}^{w;m,n;\alpha}_{a+;p,q;\beta}$ whose kernel involve the Fox's $H-$ function. By suitably specializing the coefficients and the parameters in these functions we can get a large number of (new and known) interesting expressions for the composition formulae which occur rather frequently in many problems of engineering and mathematical analysis but here we can mention only those which follow as particular cases of the Srivastava et al.\cite{ZT}.

 \end{quotation}
 \end{center}

\vskip 3mm


 \section{INTRODUCTION}
  Composition formulas involving many known fractional integral operators play important
r\^{o}les in themselves and in their diverse applications. Various expressions for the composition functions have been established by a large number of authors in many different ways (see, e.g., \cite{1, 38, 311, HM3, ZT, 321, 324}).\vskip 3mm

 Motivated by above mentioned works and the references cited therein, here we make use of the following Riemann-Liouville fractional integral operator $ I^{\mu}_{a+} $ and the Riemann-Liouville fractional derivative operator $D^{\mu}_{a+}$, which are defined by (see, for details, \cite{38}, \cite{39} and \cite{311}):
\begin{equation}
(I^{\mu}_{a+}f)(x)=\frac{1}{\Gamma(\mu)}\int_{a}^{x}\frac{f(t)}{(x-t)^{1-\mu}}\;dt \qquad \big(\Re(\mu)>0\big) \label{31}
\end{equation}
and
\begin{equation}
(D^{\mu}_{a+}f)(x)=\left(\frac{d}{dx}\right)^{n}(I^{n-\mu}_{a+}f)(x)\qquad \big(\Re(\mu)>0;\; n=[\Re(\mu)]+1), \label{32}
\end{equation}
where $ [x] $ denotes the greatest integer in the real number $ x $. Hilfer \cite{ZT} generalized the operator in (\ref{32}) and defined a general fractional derivative operator $ D^{\mu,\nu}_{a+} $ of order $ 0<\mu<1 $ and type
$ 0\leqq \nu\leqq 1 $ with respect to $ x $ as follows:
\begin{equation}
(D^{\mu,\nu}_{a+}f)(x)=\left(I^{\nu(1-\mu)}_{a+}\;\frac{d}{dx}\left(I^{(1-\nu)(1-\mu)}_{a+}f\right)\right)(x).\label{33}
\end{equation}

\begin{remark} The generalization in (\ref{33}) yields the classical Riemann-Liouville fractional derivative operator $ D^\mu_{a+} $ when $ \nu=0 $. When $ \nu=1,$ (\ref{33}) reduces to the fractional derivative operator introduced by Joseph Liouville, which is often attributed now-a-days to Caputo (see \cite{38} and \cite{324}; see also \cite{34}).
\end{remark}
Various general families of operators of fractional integration involving the $H$-function and its extensions including those in two and more variables were considered extensively by Srivastava and Saxena (see, for details, \cite[Sections 6 to 9]{321}). Here, in our present investigation, we find it to be convenient to use the following special case of one of these general
 families of fractional integral operators with the $H$-function in their kernels (see \cite[p. 15, eq. (6.3)]{321} and the references cited therein):\\

  \begin{equation}
  (I^{\gamma , \mu}_{a+}f)(x) = \frac{(x-a)^{-\mu-\gamma}}{\Gamma (\mu)}\int_{a}^{x}\frac{t^{\gamma}f(t)}{(x-t)^{1-\mu}}dt
  \quad(\Re(\gamma)>-1,\quad \Re(\mu)> 0)\label{612}
  \end{equation}

\begin{equation}
\left(\mathcal{H}^{w;m,n;\alpha}_{a+;p,q;\beta}\;\varphi\right)(x):= \int_{a}^{x}(x-t)^{\beta-1}\;H^{m,n}_{p,q}[w(x-t)^{\alpha}]\varphi(t)dt \label{36}
\end{equation}
$$\Bigg(\Re(\beta)>0;\; w \in  \mathbb{C}\setminus\{0\};\; 1\leqq m \leqq q;\; 0 \leqq n \leqq p;\; \Re(\beta)+\min\limits_{1\leqq j \leqq m}
\left\{\Re\left(\frac{\alpha b_j}{\beta_j}\right)\right\}>0 \Bigg).$$

 \section{Main Results}
 In this section, we derive composition formulae for the fractional integral operator $ I^{\gamma , \mu}_{a+} $, fractional integral operators with the $H$-function in their kernels, the Riemann-Liouville fractional derivative operator $D^{\mu}_{a+}$ and  a general fractional derivative operator $ D^{\mu,\nu}_{a+} $ of order $ 0<\mu<1 $ defined by \eqref{612}, \eqref{36}, \eqref{32} and \eqref{33}, respectively. Our main results are asserted by Theorems 1, 2, 3 and 4 below.\\

 \begin{theorem} Let $\Re(\gamma)>-1$ and $\Re(\mu)> 0, \Re(\beta)>0;\; w \in  \mathbb{C}\setminus\{0\};\; 1\leqq m \leqq q;\; 0 \leqq n \leqq p;\; \Re(\beta)+\min\limits_{1\leqq j \leqq m}
\left\{\Re\left(\frac{\alpha b_j}{\beta_j}\right)\right\}>0$, provided $ \left|\frac{x-u}{x-a}\right|<1.$
Then there holds the following formula:
 \begin{align}
 \left(\mathcal{H}^{w;m,n;\alpha}_{a+;p,q;\beta}I^{\gamma , \mu}_{a+}\phi \right)(x)= &\int_{a}^{x}u^\gamma (x-u)^{\mu+\beta-1}\frac{(x-a)^{-\mu-\gamma}}{\Gamma(\mu+\gamma)}\nonumber\\ \nonumber \\& H^{0,1:m,n;1,1}_{1,1:p,q;1,1}\left[\begin{array}{cc}
 w(x-u)^\alpha\\-\frac{x-u}{x-a}
 \end{array}\left|\begin{array}{cccc}
 (1-\beta;\alpha,1):& (a_j,\alpha_j)_{1,p}; (1-\mu-\gamma,1)\\ \\ (1-\mu-\beta;\alpha,1):&(b_j,\beta_j)_{1,q};(0,1)
 \end{array}\right]\right.\nonumber \\ &\phi(u)du\label{6.11}
 \end{align}
 where $ \mathcal{H}^{w;m,n;\alpha}_{a+;p,q;\beta} $ is given by (\ref{36}) and $ I^{\gamma , \mu}_{a+} $ is given by (\ref{612}).
 \end{theorem}
 \begin{proof}
 Let $\mathcal{C}$ be the left-hand side of \ref{6.11}. Using (\ref{36}) and (\ref{612}) respectively, we have
 \begin{equation}
 \mathcal{C}=\int_{a}^{x}(x-t)^{\beta-1}\;H^{m,n}_{p,q}[w(x-t)^{\alpha}]\frac{(t-a)^{-\mu-\gamma}}{\Gamma (\mu)}\int_{a}^{t}\frac{u^{\gamma}}{(t-u)^{1-\mu}}\phi(u)du dt\label{6.12}
 \end{equation}
 Now, we interchange the order of $ u- $integral and $ t- $integral, which is permissible under the conditions stated, we easily arrive at the following after a little simplification:
 \begin{equation}
 \mathcal{C}=\int_{a}^{x}\frac{u^\gamma}{\Gamma(\mu)}\Delta\phi(u)du\label{6.13}
 \end{equation}
 where
 \begin{equation}
 \Delta=\int_{u}^{x}(t-u)^{\mu-1}(x-t)^{\beta-1}(t-a)^{-\mu-\gamma}H^{m,n}_{p,q}[w(x-t)^\alpha]dt\label{6.14}
 \end{equation}
 To evaluate $ \Delta $, we first replace the $H-$function occurring in it in terms of its Mellin-Barnes contour integral with the help of \cite[p.10]{HM3} and interchange the order of contour integral and $t-$integral, which is permissible under the given conditions.\\
 The above equation (\ref{6.14}) now takes the following form after a little simplification:
 \begin{equation}
 \Delta=\frac{1}{2 \pi i}\int_{L}\varphi(\xi)w^\xi\int_{u}^{x}(t-u)^{\mu-1}(x-t)^{\beta+\alpha\xi-1}(t-a)^{-\mu-\gamma}dt d\xi\label{6.15}
 \end{equation}
 On setting $ z= \frac{x-t}{x-u} $ in the $ t- $integral involved in (\ref{6.15}) and evaluating the resulting $ z- $integral with the help of the known result \cite[p.286, eq.(3.197(3))]{grad}, we arrive at the following result after a little simplication:
 \begin{align}
 \Delta=\frac{1}{2 \pi i}\int_L\varphi(\xi)w^\xi &(x-u)^{\mu+\beta+\alpha\xi-1} (x-a)^{-\mu-\gamma} B(\mu,\beta+\alpha \xi)\nonumber\\&{}_2F_1\left[\mu+\gamma,\beta+\alpha \xi;\mu+\beta+\alpha \xi;-\frac{x-u}{x-a}\right] d\xi\label{6.16}
 \end{align}
 Now, writing $ {}_2F_1 $ in terms of its contour and reinterpreting the above equation(\ref{6.16}) in terms of the $ H- $function of two variables and on substituting the value of $ \Delta $ thus obtained, in (\ref{6.13}), we easily arrive at the desired result (\ref{6.11}) after a little simplification.
 \end{proof}
\begin{theorem} Let $\Re(\gamma)>-1$ and $\Re(\mu)> 0, \Re(\beta)>0;\; w \in  \mathbb{C}\setminus\{0\};\; 1\leqq m \leqq q;\; 0 \leqq n \leqq p;\; \Re(\beta)+\min\limits_{1\leqq j \leqq m}
\left\{\Re\left(\frac{\alpha b_j}{\beta_j}\right)\right\}>0$, provided $ \left|\frac{x-u}{x}\right|<1 $
Then there holds the following formula:
 \begin{align}
 \left(I^{\gamma , \mu}_{a+}\mathcal{H}^{w;m,n;\alpha}_{a+;p,q;\beta}\phi \right)&(x)= \frac{x^\gamma(x-a)^{-\mu-\gamma}}{\Gamma(\mu)\Gamma(-\gamma)}\int_{a}^{x}x^\gamma (x-u)^{\mu+\beta-1}\nonumber\\ \nonumber \\& H^{0,0:m,n+1;1,2}_{0,1:p+1,q;2,1}\left[\begin{array}{cc}
 w(x-u)^\alpha\\-\frac{x-u}{x}
 \end{array}\left|\begin{array}{cccc}
 -:(1-\beta,\alpha), (a_j,\alpha_j)_{1,p}; (1+\gamma,1),(1-\mu,1)\\ \\ (1-\mu-\beta;\alpha,1):(b_j,\beta_j)_{1,q};(0,1)
 \end{array}\right]\right.\nonumber \\ &\phi(u)du\label{6.17}
 \end{align}
 where $\mathcal{H}^{w;m,n;\alpha}_{a+;p,q;\beta} $ is given by (\ref{36}) and $ I^{\gamma , \mu}_{a+} $ is given by (\ref{612}).
 \end{theorem}
 \begin{proof}
 To prove (\ref{6.17}), let $\mathcal{K}$ be the left-hand side of \ref{6.13}. By using (\ref{36}) and (\ref{612}) respectively, we have
 \begin{equation}
 \mathcal{K}=\frac{(x-a)^{-\mu-\gamma}}{\Gamma(\mu)}\int_{a}^{x}t^{\gamma}(x-t)^{\mu-1}\int_{a}^{t}(t-u)^{\beta-1}H^{m,n}_{p,q}[w(x-t)^{\alpha}]\phi(u)du dt \label{6.18}
 \end{equation}
 Now, we interchange the order of $ u- $integral and $ t- $integral, which is permissible under the conditions stated, we easily arrive at the following after a little simplification:
 \begin{equation}
 \mathcal{K}=\frac{(x-a)^{-\mu-\gamma}}{\Gamma(\mu)}\int_{a}^{x}\Delta\phi(u)du\label{6.19}
 \end{equation}
 where
 \begin{equation}
 \Delta=\int_{u}^{x}t^\gamma(t-u)^{\beta-1}(x-t)^{\mu-1}H^{m,n}_{p,q}[w(x-t)^\alpha]dt\label{6.20}
 \end{equation}
 To evaluate $ \Delta $, we first replace the $ H- $ function occuring in it in terms of its Mellin-Barnes contour integral with the help of \cite[p.10]{HM3} and interchange the order of contour integral and $t-$integral, which is permissible under the given conditions.\\
 The above equation (\ref{6.20}) now takes the following form after a little simplification:
 \begin{equation}
 \Delta=\frac{1}{2 \pi i}\int_{L}\varphi(\xi)w^\xi\int_{u}^{x}t^\gamma(t-u)^{\beta+\alpha\xi-1}(x-t)^{\mu-1}dt d\xi\label{6.21}
 \end{equation}
 On setting $ z= \frac{x-t}{x-u} $ in the $ t- $integral involved in (\ref{6.21}) and evaluating the resulting $ z- $integral with the help of the known result \cite[p.286, eq.3.197(3)]{grad}, we arrive at the following result after a little simplication:
 \begin{align}
 \Delta=\frac{1}{2 \pi i}\int_L\varphi(\xi)w^\xi x^\gamma(x-u)^{\mu+\beta+\alpha\xi-1}& B(\mu,\beta+\alpha \xi)\nonumber \\ &{}_2F_1\left[-\gamma,\mu;\mu+\beta+\alpha \xi;-\frac{x-u}{x}\right]d\xi \label{6.22}
 \end{align}
 Now, expressing $ {}_2F_{1} $ in its contour form and reinterpreting the above equation(\ref{6.22}) in terms of the $ H- $function of two variables and on substituting the value of $ \Delta $ thus obtained, in (\ref{6.19}), we easily arrive at the desired result (\ref{6.17}) after a little simplification.
 \end{proof}

 \begin{theorem}
 Let $\Re(\mu)>0;\; n=[\Re(\mu)]+1$ and $\Re(\beta)>0;\; w \in  \mathbb{C}\setminus\{0\};\; 1\leqq m \leqq q;\; 0 \leqq n \leqq p;\; \Re(\beta)+\min\limits_{1\leqq j \leqq m}
\left\{\Re\left(\frac{\alpha b_j}{\beta_j}\right)\right\}>0$.
Then there holds the following formula:
 \begin{equation}
 D^{\mu}_{a+}\left(\mathcal{H}^{w;m,n;\alpha}_{a+;p,q;\beta}\phi\right)(x)=\left(\mathcal {H}^{w;m,n+2;\alpha}_{a+;p+2,q+2;\beta-\mu}\phi\right)(x)\label{6.24}
 \end{equation}
 where $ D^{\mu}_{a+} $ is given by (\ref{32}) and $ \mathcal{H}^{w;m,n;\alpha}_{a+;p,q;\beta} $ is given by (\ref{36}).
 \end{theorem}
 \begin{proof}
 To prove (\ref{6.24}) we make use of definition (\ref{32}) of $ D^{\mu}_{a+} $ involved in the left hand side of (\ref{6.24}), we get
 \begin{equation}
 D^{\mu}_{a+}\left(\mathcal{H}^{w;m,n;\alpha}_{a+;p,q;\beta}\phi\right)(x)=\left(\frac{d}{dx}\right)^n I^{n-\lambda}_{a+}\left(\mathcal{H}^{w;m,n;\alpha}_{a+;p,q;\beta}\phi\right)\label{step11}
 \end{equation}
 With the help of result (\ref{6.23}) the above equation (\ref{step11}) takes the following form:
 \begin{align}
 D^{\mu}_{a+}\left(\mathcal{H}^{w;m,n;\alpha}_{a+;p,q;\beta}\phi\right)(x)&=\left(\frac{d}{dx}\right)^n \left(\mathcal{H}^{w;m,n+1;\alpha}_{a+;p+1,q+1;\beta+n-\lambda}\phi\right)\nonumber \\ \nonumber \\ &=\left(\frac{d}{dx}\right)^n\int_{a}^{x}(x-t)^{\beta-1+n-\lambda}H^{m,n+1}_{p+1,q+1}[w(x-t)^\alpha]\phi(t)dt\nonumber \\ \nonumber \\ &=\int_{a}^{x}(x-t)^{\beta-1-\lambda}H^{m,n+2}_{p+2,q+2}[w(x-t)^\alpha]\phi(t)dt\nonumber \\ \nonumber \\ &=\left(\mathcal {H}^{w;m,n+2;\alpha}_{a+;p+2,q+2;\beta-\mu}\phi\right)(x)\label{step12}
 \end{align}

 \end{proof}
 \begin{remark}
 On similar lines we can prove the following:
 \begin{equation}
 \left(\mathcal{H}^{w;m,n;\alpha}_{a+;p,q;\beta}D^{\mu}_{a+}\phi\right)(x)=\left(\mathcal {H}^{w;m,n+2;\alpha}_{a+;p+2,q+2;\beta-\mu}\phi\right)(x)
 \end{equation}
 \end{remark}
 \begin{theorem}
 Let $ 0<\mu<1; 0\leqq \nu\leqq 1 $  and $\Re(\beta)>0;\; w \in  \mathbb{C}\setminus\{0\};\; 1\leqq m \leqq q;\; 0 \leqq n \leqq p;\; \Re(\beta)+\min\limits_{1\leqq j \leqq m}
\left\{\Re\left(\frac{\alpha b_j}{\beta_j}\right)\right\}>0$.
Then there holds the following formula:
 \begin{equation}
 D^{\mu,\nu}_{a+}\left(\mathcal{H}^{w;m,n;\alpha}_{a+;p,q;\beta}\phi\right)(x)=\left(\mathcal {H}^{w;m,n+3;\alpha}_{a+;p+3,q+3;\beta-\mu}\phi\right)(x)\label{formula5}
 \end{equation}
 where $ D^{\mu,\nu}_{a+} $ is given by (\ref{33}) and $ \mathcal{H}^{w;m,n;\alpha}_{a+;p,q;\beta} $ is given by (\ref{36}).
 \end{theorem}
 \begin{proof}
 Making use of the composition relationships asserted by (\ref{6.24}) and (\ref{6.23}), we find that
 \begin{equation}
 D^{\mu+\nu-\mu \nu}_{a+}\left(\mathcal{H}^{w;m,n;\alpha}_{a+;p,q;\beta}\phi\right)(x)=\left(\mathcal {H}^{w;m,n+1;\alpha}_{a+;p+1,q+3;\beta-\mu-\nu+\mu \nu}\phi\right)(x)\label{6.25}
 \end{equation}
 and
 \begin{align}
 D^{\mu,\nu}_{a+}\left(\mathcal{H}^{w;m,n;\alpha}_{a+;p,q;\beta}\phi\right)(x)& =I^{\nu(1-\mu)}_{a+}D^{\mu+\nu-\mu \nu}_{a+}\left(\mathcal{H}^{w;m,n;\alpha}_{a+;p,q;\beta}\phi\right)(x)\nonumber \\ \nonumber \\&=I^{\nu(1-\mu)}_{a+}\left(\mathcal {H}^{w;m,n+1;\alpha}_{a+;p+1,q+3;\beta-\mu-\nu+\mu \nu}\phi\right)(x)\nonumber \\ \nonumber \\&=\left(\mathcal {H}^{w;m,n+3;\alpha}_{a+;p+3,q+3;\beta-\mu}\phi\right)(x)\label{formula5.1}
 \end{align}
 which would complete the proof of (\ref{formula5}).
 \end{proof}
  \begin{remark}If we reduce the $ H-$function involved in the integral operator $\mathcal{H}^{w;m,n;\alpha}_{a+;p,q;\beta}$ to Mittag-Leffler function in  (\ref{6.23}), (\ref{6.24}) and (\ref{formula5}) we get the results obtained by Srivastava et al.\cite[p. 7, eq.(2.23), (2.24) and (2.25)]{ZT} respectively.
  \end{remark}
 \section{ Concluding Remarks and Observations}
  It is important to mention here that, whenever the integral operator $\mathcal{H}^{w;m,n;\alpha}_{a+;p,q;\beta}$
reduces to the Mittag-Leffler function, modified Bessel functions and other related special functions, the results become relatively more important from the application viewpoint in statistics and other disciplines. The $H$-function can be expressed in terms of the Gauss hypergeometric function and other related hypergeometric functions. Therefore, the present composition formulae play important r\^{o}les in the theory of special functions.

\vskip 2mm

We conclude our present investigation by presenting a special case out of many by  setting $ \gamma=0 $ in (\ref{612}), under the various parametric constraints listed already with the definition (\ref{36}), the following Corollary is obtained as special cases of  (\ref{6.11}) and (\ref{6.17}):
 \begin{corollary}
 Let $\Re(\mu)> 0, \Re(\beta)>0;\; w \in  \mathbb{C}\setminus\{0\};\; 1\leqq m \leqq q;\; 0 \leqq n \leqq p;\; \Re(\beta)+\min\limits_{1\leqq j \leqq m}
\left\{\Re\left(\frac{\alpha b_j}{\beta_j}\right)\right\}>0$
 \begin{equation}
 \left(\mathcal{H}^{w;m,n;\alpha}_{a+;p,q;\beta}I^{\mu}_{a+}\phi \right)(x)=\left(\mathcal {H}^{w;m,n+1;\alpha}_{a+;p+1,q+1;\beta+\mu}\phi\right)(x)\label{6.23}
 \end{equation}
 \end{corollary}
 \begin{proof}
 To prove (\ref{6.23}), we first express both $ \mathcal{H}^{w;m,n;\alpha}_{a+;p,q;\beta}$ and $I^{\mu}_{a+} $ involved in its left hand side, in the integral form with the help of (\ref{36}) and (\ref{31}) respectively, we have
 \begin{equation}
 \left(\mathcal{H}^{w;m,n;\alpha}_{a+;p,q;\beta}I^{\mu}_{a+}\phi \right)(x)=\int_{a}^{x}(x-t)^{\beta-1}H^{m,n}_{p,q}[w(x-t)^\alpha]\frac{1}{\Gamma(\mu)}\int_{a}^{t}(t-u)^{\mu-1}\phi(u)du dt\label{step1}
 \end{equation}
 Next, we change the order of $ u- $integral and $ t- $integral, which is permissible under the conditions stated, we easily arrive at the following after a little simplification:
 \begin{equation}
 \left(\mathcal{H}^{w;m,n;\alpha}_{a+;p,q;\beta}I^{\mu}_{a+}\phi \right)(x)=\frac{1}{\Gamma(\mu)}\int_{a}^{x}\Delta\phi(u)du\label{step2}
 \end{equation}
 where
 \begin{equation}
 \Delta=\int_{u}^{x}(t-u)^{\mu-1}(x-t)^{\beta-1}H^{m,n}_{p,q}[w(x-t)^\alpha]dt\label{step3}
 \end{equation}
 To evaluate $ \Delta $, we first replace the $ H- $ function occuring in it in terms of its Mellin-Barnes contour integral with the help of \cite[p.10]{HM3} and interchange the order of contour integral and $t-$integral, which is permissible under the given conditions.\\
 The above equation (\ref{step3}) now takes the following form after a little simplification:
 \begin{equation}
 \Delta=\frac{1}{2 \pi i}\int_{L}\varphi(\xi)w^\xi\int_{u}^{x}(t-u)^{\mu-1}(x-t)^{\beta+\alpha \xi-1}dt d\xi\label{step4}
 \end{equation}
 On setting $ z= \frac{x-t}{x-u} $ in the $ t- $integral involved in (\ref{step4}) and evaluating the resulting $ z- $integral, we arrive at the following result after a little simplication:
 \begin{align}
 \Delta=\frac{1}{2 \pi i}\int_L\varphi(\xi)w^\xi (x-u)^{\mu+\beta+\alpha\xi-1}\frac{\Gamma(\beta+\alpha \xi)}{\Gamma(\mu+\beta+\alpha\xi)}d\xi\label{step5}
 \end{align}
 Now, reinterpreting the above equation(\ref{step5}) in terms of the $ H- $function and on substituting the value of $ \Delta $ thus obtained, in (\ref{step2}), we easily arrive at the desired result (\ref{6.23}) after a little simplification.
 \end{proof}
 \vskip 3mm
 on similar lines we can also prove the following:
  \begin{corollary} Let $\Re(\mu)> 0, \Re(\beta)>0;\; w \in  \mathbb{C}\setminus\{0\};\; 1\leqq m \leqq q;\; 0 \leqq n \leqq p;\; \Re(\beta)+\min\limits_{1\leqq j \leqq m}
\left\{\Re\left(\frac{\alpha b_j}{\beta_j}\right)\right\}>0$ 
Then there holds the following formula:
 \begin{equation}
 \left(I^{\mu}_{a+}\mathcal{H}^{w;m,n;\alpha}_{a+;p,q;\beta}\phi \right)(x)=\left(\mathcal {H}^{w;m,n+1;\alpha}_{a+;p+1,q+1;\beta+\mu}\phi\right)(x)
 \end{equation}
 \end{corollary}
 We now turn to one of the fundamentally important transcendental functions, generalization of the modified Bessel function of the third kind or Macdonald function which will be represented in the following form:\cite[p.152, eq.(1.2); p.155, eq.(2.6)]{HJ}
  \begin{align}
  \lambda^{(\eta)}_{\mu,\nu}[z]&=\dfrac{\eta}{\Gamma(\mu+1-1/\eta)}\int\limits_{1}^{\infty}(t^\eta-1)^{\mu-1/\eta}t^{\nu}e^{-zt}dt \nonumber\\ &=H^{2,0}_{1,2}\left[z\left|\begin{array}{cc}
  (1-(\nu+1)/\eta,1/\eta)\\(0,1),(-\mu-\nu/\eta,1/\eta)
  \end{array}\right]\right.\label{lambda1}
  \end{align}
  $$\left( \eta > 0;\quad \Re(\mu)>1/\eta-1; \quad \nu\in\Re;\quad \Re(z)>0 \right)$$
  The function in (\ref{lambda1}) was introduced by Kilbas et al.\cite{KS}. Such a function was used by Bonilla et al.\cite{BK} to solve some homogeneous differential equations of fractional order and Volterra integral equations.\\
  \begin{corollary}
  Let $ \Re(\mu)> 0, \Re(\beta)>0;\; w \in  \mathbb{C}\setminus\{0\}; \left( \eta > 0;\quad \Re(\mu)>1/\eta-1; \quad \nu\in\Re; \right)$\\ Then there holds the following formula:
  \begin{equation}
  \left(I^{\mu}_{a+}\mathcal{H}^{w;2,0;\alpha}_{a+;1,2;\beta}\phi \right)(x)=\left(\mathcal {H}^{w;2,1;\alpha}_{a+;2,3;\beta+\mu}\phi\right)(x)
  \end{equation}

  \end{corollary}
 \begin{corollary}
 Let $ \Re(\mu)> 0, \Re(\beta)>0;\; w \in  \mathbb{C}\setminus\{0\}; \left( \eta > 0;\quad \Re(\mu)>1/\eta-1; \quad \nu\in\Re; \right)$\\ Then there holds the following formula:
   \begin{equation}
   \left(\mathcal{H}^{w;2,0;\alpha}_{a+;1,2;\beta}I^{\mu}_{a+}\phi \right)(x)=\left(\mathcal {H}^{w;2,1;\alpha}_{a+;2,3;\beta+\mu}\phi\right)(x)
   \end{equation}
 \end{corollary}
  Corollaries 3 and 4 can be proved on similar lines of corollary 1.
  \begin{corollary}
  Let $ \Re(\mu)> 0, \Re(\beta)>0;\; w \in  \mathbb{C}\setminus\{0\}; \left( \eta > 0;\quad \Re(\mu)>1/\eta-1; \quad \nu\in\Re; \right)$\\ Then there holds the following formula:
  \begin{equation}
  \left(\mathcal{H}^{w;2,0;\alpha}_{a+;1,2;\beta}D^{\mu}_{a+}\phi\right)(x)=\left(\mathcal {H}^{w;2,2;\alpha}_{a+;3,4;\beta-\mu}\phi\right)(x)
  \end{equation}
  \end{corollary}
  \begin{corollary}
  Let $ \Re(\mu)> 0, \Re(\beta)>0;\; w \in  \mathbb{C}\setminus\{0\}; \left( \eta > 0;\quad \Re(\mu)>1/\eta-1; \quad \nu\in\Re; \right)$\\ Then there holds the following formula:
  \begin{equation}
    \left(D^{\mu}_{a+}\mathcal{H}^{w;2,0;\alpha}_{a+;1,2;\beta}\phi\right)(x)=\left(\mathcal {H}^{w;2,2;\alpha}_{a+;3,4;\beta-\mu}\phi\right)(x)
    \end{equation}
  \end{corollary}
   Corollaries 5 and 6 can be proved on similar lines of theorem 3.

  \end{document}